\documentclass[11 pt]{amsart}

\usepackage[margin=1.25in]{geometry}
\usepackage{setspace}
\usepackage{parskip}
\usepackage{amsaddr}

\usepackage{amssymb}
\usepackage{graphicx}
\usepackage{asymptote}
\usepackage[hidelinks]{hyperref}

\theoremstyle{plain}
\newtheorem{theorem}{Theorem}
\newtheorem{lemma}{Lemma}[section]

\newtheorem{prop}{Proposition}[section]

\numberwithin{equation}{section}

\newcommand{\E}[1]{\mathbf{E} \, #1}
\newcommand{\pr}[1]{\mathbf{Pr}\left[#1\right]}
\newcommand{\ind}[1]{\mathbf{1}_{\{ #1 \}}}
\newcommand{\eps}{\varepsilon}
\newcommand{\N}{\mathbb{N}}
\newcommand{\Z}{\mathbb{Z}}

\newcommand{\lm}{\lambda}
\newcommand{\G}{\mathbf{G}}

\DeclareMathOperator{\pa}{pa}
\DeclareMathOperator{\ch}{ch}

\begin{document}

\title{On the local geometry of graphs in terms of their spectra}

\author{Brice Huang}
\address{Department of Mathematics, Massachusetts Institute of Technology, \\\ Cambridge, MA, USA.}
\email{bmhuang@mit.edu}

\author{Mustazee Rahman}
\address{Department of Mathematics, Massachusetts Institute of Technology, \\ Cambridge, MA, USA and \\
	Department of Mathematics, KTH Royal Institute of Technology, \\ Stockholm, Sweden.}
\email{mustazee@gmail.com}

\date{}


\begin{abstract}
	\normalsize
	In this paper we consider the relation between the spectrum and the number of short cycles in large graphs.
	Suppose $G_1, G_2, G_3, \ldots$ is a sequence of finite and connected graphs that share a
	common universal cover $T$ and such that the proportion of eigenvalues of $G_n$ that
	lie within the support of the spectrum of $T$ tends to 1 in the large $n$ limit. This is a weak
	notion of being Ramanujan. We prove such a sequence of graphs is asymptotically locally tree-like.
	This is deduced by way of an analogous theorem proved for certain infinite sofic graphs and unimodular networks,
	which extends results for regular graphs and certain infinite Cayley graphs.
\end{abstract}
\maketitle

\newpage
\section{Introduction} \label{sec:intro}

This paper is about how the spectrum of a big, bounded degree graph determines its
local geometry around typical vertices. For a finite and connected graph $G$, let
$$\lm_1(G) > \lm_2(G) \geq \lm_3(G) \geq \cdots$$ be the eigenvalues of its adjacency matrix.
Let $T$ be the universal cover tree of $G$ and denote by $\rho(T)$ its spectral radius, which
is the operator norm of the adjacency matrix of $T$ acting on $\ell^2(T)$. If $G$ is
$d$-regular then $\lm_1(G) = d$ and $\rho(T) = 2\sqrt{d-1}$, $T$ being the $d$-regular tree.

It is easy to see that $\lm_1(G) \geq \rho(T)$. Various extensions of the Alon-Boppana
Theorem state that for every $\eps > 0$, a positive proportion of the eigenvalues of $G$ lie
outside the interval $[-\rho(T) + \eps, \rho(T) - \eps]$ independently of the size of $G$;
see e.g.~\cite{Ci, Gr, GZ, Ra, Se}. What happens when the eigenvalues actually concentrate
within $[-\rho(T), \rho(T)]$?

A graph $G$ is Ramanujan if $|\lm_i(G)| \leq \rho(T)$ for every $i \geq 2$.
It is a fairly well-understood theme that large, $d$-regular Ramanujan graphs
locally resemble the $d$-regular tree in that they contain few short cycles.
For an illustration of such results see \cite{AGV, Br2, HLW, LP, LPS} and references therein.
This relation is not as well understood for sparse, irregular graphs. We prove the
following with regards to the spectrum and local geometry.

Suppose $G_1, G_2, G_3, \ldots$ is a sequence of finite and connected graphs.
They are \emph{weakly Ramanujan} if they have a common universal cover tree $T$
and if, counting with multiplicity,
\begin{equation} \label{eqn:Ramanujan}
\frac{\# \{\text{Eigenvalues of}\; G_n\; \text{s.t.} \;\; |\lm_i(G_n)| \leq \rho(T) \}}{|G_n|}
\, \longrightarrow \, 1\;\; \text{as}\; n \to \infty.
\end{equation}
\begin{theorem} \label{thm:0}
	Consider a sequence of weakly Ramanujan graphs as in \eqref{eqn:Ramanujan}.
	Suppose that $|G_n| \to \infty$. Then the graphs are asymptotically locally tree-like in that for every $r > 0$,
	$$ \frac{\# \{\text{Vertices}\; v \in G_n \; \text{s.t.~its r-neighbourhood is a tree}\}}{|G_n|}
	\, \longrightarrow \, 1\;\; \text{as}\; n \to \infty.$$
\end{theorem}

Let us make a few remarks about Theorem \ref{thm:0}. First, if the $r$-neighbourhood of a
vertex $v \in G$ is a tree then it agrees with the $r$-neighbourhood of any vertex
$\hat{v}$ in the universal cover of $G$ that maps to $v$ under the cover map.
So roughly speaking, large weakly Ramanujan graphs locally look like their universal
covers around most vertices.

Second, it is natural in our context to assume a sequence of graphs share a common
universal cover. For one thing it is a generalization of a sequence of $d$-regular graphs,
$(a,b)$-biregular graphs, etc. But more so, it provides a way to compare the
spectrum and geometry of graphs with differing sizes on a common scale.
For example, two finite graphs with the same universal cover have the same degree distribution,
hence, also the same average and maximal degree. They also have the same maximal eigenvalue
according to a theorem in \cite{Le}. The definition of Ramanujan graphs in terms of
their universal covers was introduced in \cite{Gr} (stated also in \cite{HLW}).

The version of Theorem \ref{thm:0} for $d$-regular weakly Ramanujan graphs
has been proved in \cite{AGV}. In this case many tools, such as the Green's
function and spectral measure of the $d$-regular tree, are available in precise form.
This is lacking for general universal covers where even the computation of the
spectral radius is difficult (although an algorithm is provided in \cite{Na} and
various bounds are given in \cite{H, Ra}).

Theorem \ref{thm:0} is also motivated by the classical work of Kesten which relates
the geometry of groups to their spectra \cite{Ks}. Namely, the $d$-regular Cayley graph of a
finitely generated group is amenable if and only if its spectral radius is the degree $d$.
In the other direction, what happens to the geometry when the spectrum is ``as small as possible",
which is to say that the spectral radius is $2 \sqrt{d-1}$? One can ask this question for
finite graphs, bounded degree sofic graphs, or more generally, for unimodular networks.
Theorem \ref{thm:0} considers a finitary case and Theorem \ref{thm:1} below certain sofic
graphs. The proof of Theorem \ref{thm:1} also applies to analogous unimodular networks.

The assumption in Theorem \ref{thm:0} that $|G_n| \to \infty$ is necessary.
For example, if all the graphs are equal to a common cyclic graph then the sequence is
weakly Ramanujan. (The universal cover is $\mathbb{Z}$ with $\rho(\mathbb{Z}) = 2$
and all the eigenvalues lie in the interval $[-2,2]$.) However, this is the only obstruction as
being weakly Ramanujan implies $|G_n| \to \infty$ so long as the common average degree
of the graphs is larger than 2. This follows from the following theorem, which
asserts that $\lm_1(G_n) > \rho(T)$ when the common average degree is more than 2.

\begin{theorem} \label{thm:2}
Let $G$ be a finite and connected graph with universal cover $T$.
Then $\lm_1(G) = \rho(T)$ if and only if $G$ has at most one cycle or, equivalently,
if and only if the average degree of $G$ is at most 2.
\end{theorem}

This theorem is somewhat of an analogue of Theorem \ref{thm:0} for a single finite graph.
The proof turns out to be a bit delicate. For instance, consider the bowtie graph $G$
obtained by gluing together two triangles at a common vertex. It has
$\lm_1(G) = (1+ \sqrt{17})/2$ and $\rho(T) =(\sqrt{3} +\sqrt{11})/2 $.
The spectral gap is about $0.04$ and the average degree is also smaller
than $\rho(T)$. In general, the spectral gap can be arbitrarily small
for graphs formed by gluing together two large cycles at a common vertex.
The proof of Theorem \ref{thm:2} is based on orienting and weighting the edges
of a finite graph in a way that allow to compare the norm of its adjacency matrix
with the cover's via the Rayleigh variational principle. We have learned it is related
to ``Gabber lemma", see \cite{GH}.

Theorem \ref{thm:0} is proved in the following section. The idea behind the proof is to
reduce the theorem to an analogous theorem about certain infinite random rooted graphs
(sofic graphs and unimodular networks) by using the notion of local convergence of graphs.
The key result of the paper is a proof of the analogue of Theorem \ref{thm:0} for these infinite
graphs, which is Theorem \ref{thm:1} below, proved in Section \ref{sec:infinite}.
The proof establishes a lower bound on the spectral radius of such graphs in terms of a probabilistic notion
of cycle density. Theorem \ref{thm:2} is then proved in Section \ref{sec:finite}.
Section \ref{sec:conc} concludes with some questions.

\section{A reduction of Theorem \ref{thm:0}} \label{sec:thm0}
We begin with the notion of local convergence of graphs and some of its properties
used for the proof of Theorem \ref{thm:0}. A complete account, including proofs,
may be found in \cite{AL, BS, Br, El}.

Let $B_r(G,v)$ be the $r$-neighbourhood of a vertex $v$ in a graph $G$.
A sequence of finite and connected graphs $G_1, G_2, G_3, \ldots$ converges locally if the following holds.
For every $r$ and every rooted, connected graph $(H,o)$ having radius at most $r$ from the root $o$,
the ratio
$$ \frac{\# \{ \text{Vertices}\; v \in G_n \;\text{s.t.}\; B_r(G_n,v) \cong (H,o)\}}{|G_n|}\;\; \text{converges as}\; n \to \infty.$$
The isomorphism relation $\cong$ is for rooted graphs, i.e.~the isomorphism must take the root
of one to the other. Local convergence is also known as Benjamini-Schramm convergence as it
was formulated by them in \cite{BS}.

A locally convergent sequence of graphs may be represented as a random rooted graph in the following way.
Let $\mathcal{G}$ be the set of all rooted and connected graphs whose vertex sets are subsets of the integers.
Identify the graphs in $\mathcal{G}$ up to their rooted isomorphism class. The set $\mathcal{G}$ is a complete and
separable metric space with the distance between $(H,o)$ and $(H',o')$ being $2^{-r}$, where $r$ is the maximal
integer such that $B_r(H,o) \cong B_r(H',o')$. A random rooted graph is simply a Borel probability measure
on $\mathcal{G}$ or, in other words, a $\mathcal{G}$-valued random variable $(\G,o)$ that is Borel-measurable.
Given a locally convergent sequence of graphs as above, there is a random rooted graph $(\G,o)$ such that
for every $r$ and $(H,o)$ as above,
$$ \frac{\# \{ \text{Vertices}\; v \in G_n \;\text{s.t.}\; B_r(G_n,v) \cong (H, o)\}}{|G_n|}
\longrightarrow \pr{B_r(\G,o) \cong (H,o)}.$$

A random rooted graph that is obtained from a locally convergent sequence
of finite graphs is called \emph{sofic}. Examples include Cayley graphs of amenable groups
such as $\Z^d$, regular trees, as well as Cayley graphs of residually finite groups. More examples
may be found in \cite{AL, BS, Br, El}. Sophic graphs satisfy an important property
known as the \emph{mass transport principle}, as we explain.

Suppose $(\G,o)$ is sofic. Consider a bounded and measurable function $F(G,u,v)$ defined over
doubly rooted graphs such that it depends only on the double-rooted isomorphism class
of $(G,u,v)$. The mass transport principle states that
$$ \E{\sum_{v \in \G} F(\G,o,v)} = \E{\sum_{v \in \G} F(\G,v,o)}.$$
The above is readily verified for a finite graph rooted at a uniformly random vertex,
and it continues to hold in the local limit, which is why it holds for a sofic graph.
Random rooted graphs that satisfy the mass transport principle are called unimodular networks.
It is not known whether every unimodular network is a sofic graph.

Let us describe the universal cover of a sofic graph.
Recall the universal cover of a graph $G$ is the unique tree $T$ for which there
is a surjective graph homomorphism $\pi: T \to G$, called the cover map, such
that $\pi$ is locally bijective: for every $\hat{v} \in T$, $\pi$ provides a bijection
$B_1(T, \hat{v})  \overset{\pi}{\longleftrightarrow} B_1(G, \pi(\hat{v}))$. If $\pi(\hat{v}) = \pi(\hat{u})$
then the rooted graphs $(T, \hat{v}) \cong (T, \hat{u})$. Therefore, the universal
cover of a sofic graph $(\G,o)$ may be defined as its samplewise
universal cover $(\mathbf{T},\hat{o})$, where $\hat{o}$ is any vertex that is
mapped to $o$ by the cover map.

The spectral radius of a sofic graph $(\G,o)$ is defined as follows.
Let $W_k(G,o)$ be the set of closed walks of length $k$ from $o$ in a graph $G$
and denote by $|W|_k(G,0)$ its size. The spectral radius of $(\G,o)$ is
$$ \rho(\G) = \lim_{k \to \infty} \, \left( \E{|W|_{2k}(\G,o)} \right )^{1/2k}.$$
Recall that the spectral radius of a connected graph $G$ is also the exponential
growth rate of $|W|_{2k}(G,v)$ for any vertex $v$. The relation of $\rho(\G)$
to the adjacency matrix of $\G$ is that it equals the sup norm, $|| \rho(G,o)||_{\infty}$,
of the samplewise spectral radius of $(\G,o)$. It is also the largest element in the support
of the ``averaged" spectral measure of $(\G,o)$, which is a probability measure over
the reals with moments $\E{|W|_{0}(\G,o)}, \E{|W|_{1}(\G,o)}, \E{|W|_{2}(\G,o)}, \ldots$.

We prove the following regarding the spectrum and geometry of an infinite sofic graph.
\begin{theorem} \label{thm:1}
	Let $(\G,o)$ be an almost surely infinite sofic graph that is the local limit of a sequence of finite graphs
	sharing a common universal cover $T$. If $\rho(\G) = \rho(T)$ then $\G$ is isomorphic
	to $T$ almost surely.
\end{theorem}
This theorem (and its proof) may be extended to unimodular networks whose universal cover is
deterministic and quasi-transitive.

\subsection{Reducing Theorem \ref{thm:0} to sofic graphs}
The theorem may be reformulated as asserting that given a sequence of weakly
Ramanujan graphs sharing a common universal cover $T$, and with their sizes
tending to infinity, the sequence converges locally to $T$. The root of $T$
will be a random vertex whose distribution is determined by the convergent sequence,
although it does not depend on the particular sequence.

More precisely, as $T$ is the universal cover of a finite graph, it is quasi-transitive. Let
\begin{equation} \label{eqn:orbit}
\big \{ \hat{v}_1, \hat{v}_2, \ldots, \hat{v}_m  \big \}
\end{equation}
be vertices that make a set of representatives for $T / \mathrm{Aut}(T)$.
There is a probability measure $(p_1, \ldots, p_m)$ on them such that if vertex $\hat{o}$ is
chosen according to it then $(T, \hat{o})$ satisfies the mass transport principle; see \cite{AL}.
Any sequence of graphs converging locally to $T$ has limit $(T, \hat{o})$. Moreover,
if a finite graph $G$ is covered by $T$ then $p_j$ is the proportion of vertices $v \in G$
such that the vertices of $\pi^{-1}(v)$ can be sent to $\hat{v}_j$ by automorphisms of $T$
(all vertices in $\pi^{-1}(v)$ belong to the same orbit); see \cite{Ma}.

Suppose $G_1, G_2, G_3, \ldots$ is a sequence of weakly Ramanujan graphs as in
the statement of the theorem. Since they share a common universal cover,
their vertex degrees are bounded by some integer $\Delta$. By a simple
diagonalization argument (there are at most $\Delta^{r+1}$ rooted
graphs of radius $ \leq r$ and maximal degree $\leq \Delta$), the sequence is
pre-compact in the local topology. We must prove that its only limit point is $(T, \hat{o})$.

Suppose $(\G,o)$ is a limit point of the sequence. Then its universal cover is isomorphic
to $(T,\hat{o})$ almost surely. This is because the sequence has a common universal cover
and the universal cover is a continuous mapping of its base graph in the local topology; see \cite{Ma}.
Theorem \ref{thm:0} is proved if $\G$ is isomorphic to $T$ almost surely as unrooted graphs.
The mass transport principle specifies the distribution of the root as we have explained.
By Lemma \ref{lem:Ramanujan} below, $\rho(\G) = \rho(T)$. The theorem now follows from Theorem \ref{thm:1}.

\begin{lemma} \label{lem:Ramanujan}
	Let $G_1, G_2, G_3, \ldots$ be a locally convergent sequence of weakly Ramanujan graphs.
	Suppose $(\G,o)$ is its limit and $T$ is the common universal cover.
	Then $\rho(\G) = \rho(T)$.
\end{lemma}

\begin{proof}
	Let $\Delta$ be the maximal vertex degree of the graph sequence.
	Due to local convergence and bounded convergence theorem, $\E{|W|_{2k}(\G,o)}$
	is the limit of $\frac{1}{|G_n|} \sum_{v \in G_n} |W|_{2k}(G_n,v)$.
	Let $q_n$ be the proportion of eigenvalues of $G_n$ that are at most $\rho(T)$ in absolute value, so then $q_n \to 1$.
	Note that all eigenvalues of $G_n$ are bounded by $\Delta$ in absolute value. The aforementioned average
	is the trace of the $(2k)$-th power of the adjacency matrix of $G_n$, normalized by $|G_n|$. Thus,
	$$ \frac{1}{|G_n|} \sum_{v \in G_n} |W|_{2k}(G_n,v)  \leq q_n \, \rho(T)^{2k} + (1-q_n) \, \Delta^{2k}.$$
	Upon taking limits we conclude that $\rho(\G) \leq \rho(T)$.

	For the inequality in the other direction, recall the vertices $\hat{v}_1, \ldots, \hat{v}_m$ from \eqref{eqn:orbit}
	and the associated probability measure $(p_1, \ldots, p_m)$ on them. Suppose a finite
	graph $G$ has universal cover $T$. Recall $p_j$ is the proportion of vertices in $G$ that have a
	pre-image in $T$, under the cover map, which can be sent to $\hat{v}_j$ by a $T$-automorphism.
	If $\hat{v} \in T$ is mapped to $v \in G$ by the cover map then $|W|_{2k}(G,v) \geq |W|_{2k}(T, \hat{v})$.
	Indeed, the cover map provides an injection from $W_{2k}(T,\hat{v})$ into $W_{2k}(G,v)$ due to its path lifting property.
	Consequently,
	$$ \frac{1}{|G|} \sum_{v \in G} |W|_{2k}(G,v)  \geq \sum_{j=1}^m p_j\, |W|_{2k}(T,\hat{v}_j).$$

	Applying the inequality above to every $G_n$ and taking the large $n$ limit gives
	$$ \E{|W|_{2k}(\G,o)} \geq \sum_{j=1}^m p_j \, |W|_{2k}(T,\hat{v}_j).$$
	Since $\rho(T)$ is the large $k$ limit of $|W|_{2k}(T, \hat{v}_j)^{1/2k}$ for every $\hat{v}_j$, the
	inequality above implies that $\rho(\G) \geq \rho(T)$.
\end{proof}

\section{A spectral rigidity theorem} \label{sec:infinite}

Theorem \ref{thm:1} will be proved by showing that if there is an $\ell$ such that
$$\pr{o\;\text{lies in an}\; \ell-\text{cycle of}\; \G} > 0,$$
then $\rho(\G)/\rho(T) \geq 1 + \delta$ for some positive $\delta$.
This result is built up in the subsequent sections by drawing a connection between
the spectral radius of $\G$ and of $T$ in terms of the norms of certain Markov operators
associated to random walks on the fundamental group of $\G$. This connection was
established in \cite{AGV}.

\subsection{Counting walks using the fundamental group}

Consider a connected graph $H$ which may be countably infinite and may have multi-edges
and loops around its vertices. (A loop contributes degree 2 to its vertex.)
Let $\pi(H,v)$ be its fundamental group based at vertex $v$, which consists of homotopy
classes of closed walks from $v$ under the operation of concatenation. It is a free group.
(See \cite{Ma} for an account on the fundamental group of graphs and its properties mentioned here.)

Let $W_k(u,v)$ be the set of walks in $H$ of length $k$ from $u$ to $v$. Note $W_k(v,u) = W_k^{-1}(u,v)$,
where the inverse means walking in the opposite direction. The set
$$ W W^{-1} = \{ P Q^{-1}: P, Q \in W_k(u,v) \}$$
consists of closed walks from $u$ of length $2k$ and is closed under inversion.
It naturally maps into $\pi(H,u)$, and the uniform measure on it pushes forward
to a measure on the image $\overline{WW^{-1}} \subset \pi(H,u)$. Note the push-forward may
not be uniform measure on the image as different closed walks in $WW^{-1}$ may be homotopy equivalent.

Consider the random walk on $\pi(H,u)$ whose step distribution is the aforementioned
push-forward measure of $WW^{-1}$. Since $WW^{-1}$ is closed under taking inverses,
this is a symmetric random walk on the Cayley graph of the subgroup of $\pi(H,u)$ generated by $WW^{-1}$.
Denote the norm of its associated Markov operator by
\begin{equation} \label{eqn:Markov} ||M_k||(u,v).\end{equation}

Now fix a vertex $o \in H$, a path $P$ from $o$ to $u$ and another path $Q$ from $o$ to $v$.
The set $ P W_k(u,v) Q^{-1}$ consists of closed walks from $o$. Consider the random walk on $\pi(H,o)$
whose step distribution is the push-forward of the uniform measure on this set by its the natural mapping into $\pi(H,o)$.
Denote by $\sqrt{||M_k||}(u,v)$ the norm of the Markov operator of this random walk. This operator may
not be symmetric since the set $P W_k(u,v)Q^{-1}$ is not closed under taking inverses. However,
$$\sqrt{||M_k||}(u,v)^2 = ||M_k||(u,v)$$
because the norm in question is the square root of the norm
of the Markov operator for the random walk on $\pi(H,o)$ associated to the set
$$ (PW_k(u,v)Q^{-1}) (PW_k(u,v)Q^{-1})^{-1} = P WW^{-1} P^{-1}.$$
The Markov operator for $P WW^{-1} P^{-1}$ is isomorphic -- as an operator on $\ell^2(\pi(H,o))$ --
to the Markov operator for $WW^{-1}$ on $\ell^2(\pi(H,u))$. The isomorphism comes from the natural
isomorphism of groups $\pi(H,o) \leftrightarrow \pi(H,u)$. The norm of the Markov operator
for $WW^{-1}$ is $||M_k||(u,v)$.

\subsection{The counting argument}
Let $H$ be a graph as above. A \emph{purely backtracking} walk in $H$ is a closed walk that is
homotopic to the empty walk, that is, it reduces to the identity in the fundamental group of $H$.
Purely backtracking walks in $H$ from a base point $o$ are in one to one correspondence with closed
walks in the universal cover of $H$ from a base point $\hat{o}$ such that $\pi(\hat{o}) = o$ ($\pi$
being the cover map). This is due to the path lifting property of the universal cover map.

Choose an arbitrary vertex $o \in H$. Let $n$ and $k$ be arbitrary integers with $nk$ being even.
Denote by $W$ all closed walks from $o$ of length $nk$. Denote by $N$ all purely
backtracking walks from $o$ of length $nk$. The following inequality is from \cite{AGV}.
\begin{equation} \label{eqn:main}
\log |W| - \log |N| \geq \frac{1}{|N|} \sum_{P \in N} \, \sum_{j=1}^{n} -\frac{1}{2}\, \log ||M_k||(P_{(j-1)k},P_{jk})\,.
\end{equation}

The proof is based on partitioning the set $W$ in the following way.
Two walks in $W$ are equivalent if their locations coincide at the times
$0, k, 2k, \ldots, nk$. Let $W_N$ denote the set of walks in $W$ that are
equivalent to some purely backtracking walk. Observe that
$$|W_N| = \sum_{P \in N} \frac{|[P]|}{|[P]\cap N|}\,.$$

The term $|[P]| / |[P] \cap N|$ is the reciprocal of the probability that
a uniform random walk in $[P]$ is purely backtracking. The probability
can be interpreted in the following way. Consider the random walk
on $\pi(H,o)$ whose step distribution is the push forward of the uniform
measure on $[P] \mapsto \pi(H,o)$. The probability under consideration
is the one-step return probability of this random walk. It may be expressed
as $\langle M_P \rm{id},\rm{id} \rangle$, where $M_P$ is the Markov operator
of this random walk. Therefore,
$$|W| \geq \sum_{P \in N}\, \langle M_P \rm{id}, \rm{id} \rangle^{-1} \,.$$

Every $Q \in [P]$ agrees with $P$ at the times $0, k, \ldots, nk$.
This allows us to decompose $Q$ into petals as in Figure \ref{fig:1}.
\begin{figure}[hbtp]
	\centering
	\includegraphics[scale=0.5]{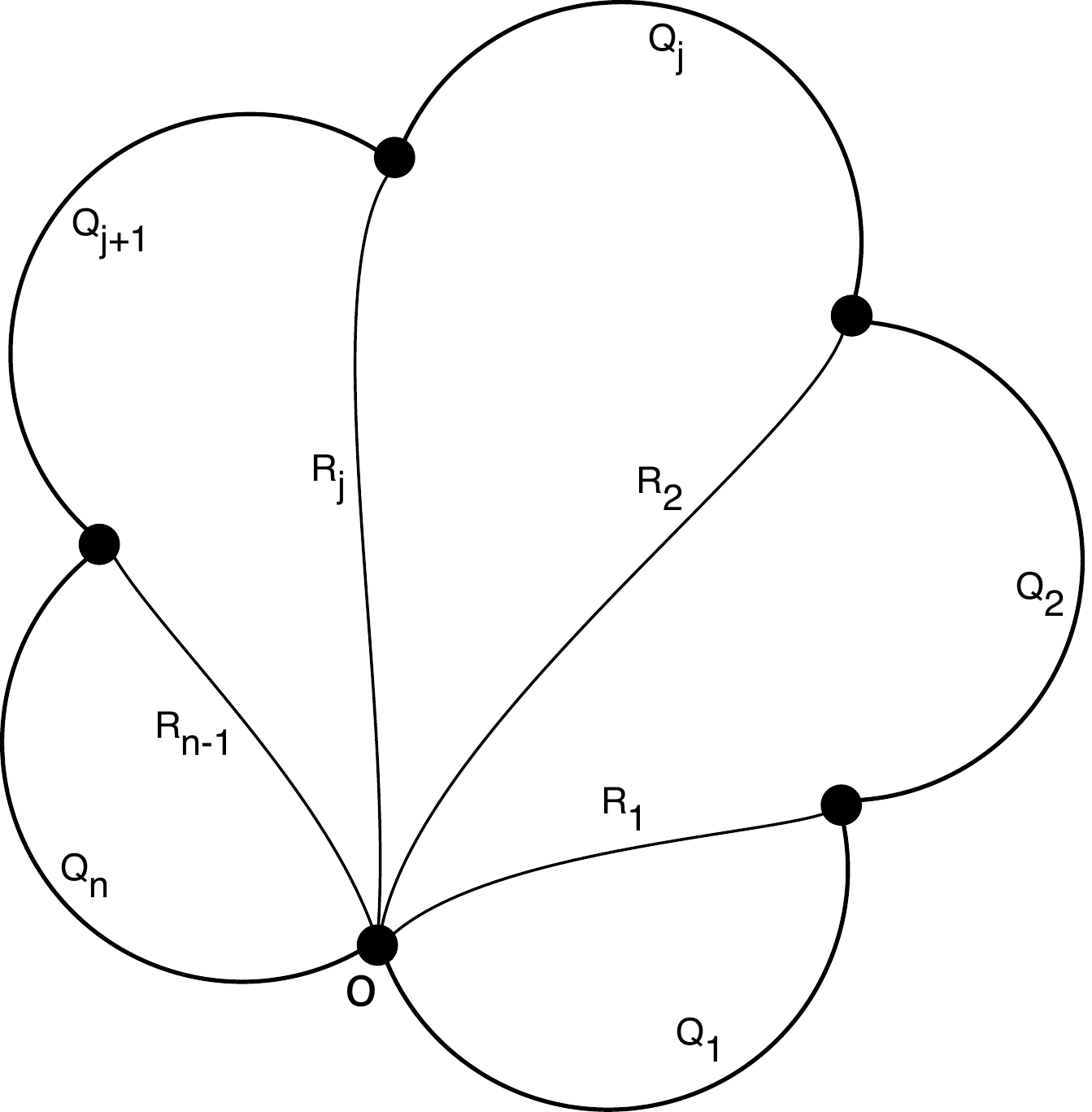}
	\caption{Decomposing a closed walk into petals.}
	\label{fig:1}
\end{figure}

Here, $Q_j$ is the segment of $Q$ from $Q_{(j-1)k} = P_{(j-1)k}$ to $Q_{jk} = P_{jk}$.
$R_j$ is a fixed path from $o$ to $P_{jk}$ chosen independently of $Q$.
The decomposition is that
$$Q = \underbrace{{(Q_1 R_1^{-1})}}_{T_1} \cdot \underbrace{(R_1 Q_2 R_2^{-1})}_{T_2} \cdots \underbrace{(R_{n-1} Q_n)}_{T_n}\,.$$
Under this decomposition, a uniformly random element $Q \in [P]$ becomes the product $T_1 \cdots T_n$,
where $T_j$ is a uniformly random element of $R_{j-1} W_k\big(P_{(j-1)k}, P_{jk}\big) R_j^{-1}$.
This uses that the locations of $Q$ are pinned at the times $0, k, \ldots, nk$.

Let $M_j$ be the Markov operator for the random walk on $\pi(H,o)$ with step distribution $T_j$.
Then $M_P = M_1 \cdots M_n$, and
$$ \langle M_P \mathrm{id}, \mathrm{id} \rangle \leq || M_P || \leq \prod_j ||M_j||\,.$$
Each $||M_j||$ equals $\sqrt{||M_k||}(P_{(j-1)k},P_{jk})$. Therefore,
$$|W| \geq \sum_{P \in N} \, \prod_{j=1}^n ||M_k||(P_{(j-1)k}, P_{jk})^{-1/2}\,.$$
Dividing the above by $|N|$, using the inequality of arithmetic-mean and geometric-mean, and then
taking the logarithm gives the inequality from \eqref{eqn:main}.

Let $(\G,o)$ be an infinite sofic graph as in the statement of Theorem \ref{thm:1}.
The mass transport principle simplifies the right hand side of \eqref{eqn:main} for
$(\G,o)$ as follows.

\begin{lemma} \label{lem:2}
In this setting the following equation holds for $j = 1, \ldots, n$.
$$\E{ \frac{1}{|N|_{nk}(\G,o)} \, \sum_{P \in N_{nk}(\G,o)}\, \log ||M_k||(P_{(j-1)k},P_{jk})} =
\E{ \sum_{P \in N_{nk}(\G,o)}\, \frac{\log ||M_k||(o,P_k)}{|N_{nk}|(\G,P_{(n-j+1)k})}}\,.$$
\end{lemma}

\begin{proof}
Consider the function
$$F(H,u,v) = \frac{1}{|N|_{nk}(H,v)} \sum_{P \in N_{nk}(H,v)} \ind{P_{(j-1)k} = u} \log ||M_k||(u, P_{jk})\,.$$
It depends on the doubly-rooted isomorphism class of $(H,u,v)$. Now,
$$\sum_{u \in H} F(H,u,v) = \frac{1}{|N|_{nk}(H,v)}\sum_{P \in N_{nk}(H,v)} \log ||M_k||(P_{(j-1)k}, P_{jk})\,.$$
On the other hand,
$$F(H,u,v) = \frac{1}{|N|_{nk}(H,v)} \sum_{P \in N_{nk}(H,u)} \ind{P_{(n-j+1)k} = v} \log ||M_k||(u, P_k)$$
because we can also sum over the walks by starting them at $u$ instead of $v$. Therefore,
$$\sum_{v \in H} F(H,u,v) = \sum_{P \in N_{nk}(H,u)}\, \frac{\log ||M_k||(u,P_k)}{|N|_{nk}(H,P_{(n-j+1)k})}\,.$$
The mass-transport principle for $(\G,o)$ states that
$$\E {\sum_{u \in \G}\, F(\G,u,o)} = \E { \sum_{v \in \G}\, F(\G,o,v)},$$
which is the equation in the statement of the lemma.
\end{proof}

\subsection{Bounds}

Applying the bound from \eqref{eqn:main} to $(\G,o)$, taking the expectation value,
applying Lemma \ref{lem:2} and then dividing by $nk$ gives
\begin{align*}
& \frac{\E{\log |W|_{nk}(\G,o)} - \E{\log|N|_{nk}(\G,o)}}{nk} \;\; \geq \\
& \qquad \qquad \E{\sum_{P \in N_{nk}(\G,o)} \frac{1}{n} \sum_{j=1}^n \frac{-(2k)^{-1}\log ||M_k||(o,P_k)}{|N|_{nk}(\G,P_{(n-j+1)k})}}.
\end{align*}

The term $-(2k)^{-1} \log ||M_k||(P_o,P_k)$ is non-negative. We would thus like to replace
each of the terms $|N|_{nk}(\G,P_{(n-j+1)k})$ by $|N|_{nk}(\G,o)$, after which the average over the parameter $j$
would be replaced by unity. Recall the universal cover of $\G$ is the non-random tree $T$ and
$W_{nk}(T,\hat{v}) = N_{nk}(\G,\pi(\hat{v}))$. Therefore, the cost of replacing $|N|_{nk}(\G,P_{(n-j+1)k})$
by $|N|_{nk}(\G,o)$ while preserving the $\geq$ inequality is given by the multiplicative factor
$$r_{nk} = \min_{i,j} \, \frac{|W_{nk}|(T, \hat{v}_i)}{|W_{nk}|(T, \hat{v}_j)}\,,$$
where $\hat{v}_1, \ldots, \hat{v}_m$ are a set of orbit representatives for $T$ as given by \eqref{eqn:orbit}.
Part 1 of the following lemma shows that $r_{nk} \geq \Delta^{-2d}$, where $d$ is the maximum
distance between any two of the $\hat{v}_i$'s and $\Delta$ is the maximal degree of $T$.

\begin{lemma}\label{lem:3}
Let $H$ be a connected graph having maximum degree at most $\Delta$.
Let $x$ and $y$ be two of its vertices having distance $d$ between them.
\begin{enumerate}
	\item $|W|_{2k}(H,y) \leq \Delta^{2d}\, |W|_{2k}(H,x)$.
	\item $|W|_{2k+2j}(H,x) \leq \Delta^{2j} \, |W|_{2k}(H,x)$.
\end{enumerate}
\end{lemma}

\begin{proof}
	Let $A$ be the adjacency matrix of $H$ acting on $\ell^2(H)$ ($H$ may be countably infinite).

	The inequality in (1) follows from the inequality in (2) upon observing that
	$|W|_{2k}(H,y) \leq |W|_{2k + 2d}(H,x)$. For the proof of (2), we have
	$|W|_{2k + 2j}(H,x) = \langle A^{2k + 2j} \delta_x, \delta_x \rangle$
	and the latter equals $\langle A^{2j} (A^{k}\delta_x), (A^{k} \delta_x) \rangle$.
	Thus,
	\begin{equation*}
	|W|_{2k + 2j}(H,x)  \leq ||A^{2j}|| \, \langle A^{k}\delta_x, A^{k} \delta_x \rangle
	\leq \Delta^{2j} \langle A^{k}\delta_x, A^{k} \delta_x \rangle = \Delta^{2j} |W|_{2k}(H,x)\,.
	\end{equation*}
\end{proof}

\begin{lemma} \label{lem:4}
	Let $\Delta$ be the maximal degree of $T$. The following inequality holds:
	\begin{equation*}
	\log \rho(\G) - \log \rho(T) \geq \; \sup_{k \geq 1} \; \frac{\E{- \log ||M_{2k}||(o,o)}}{4k \, \Delta^{2d+2k}}\,.
	\end{equation*}
\end{lemma}

\begin{proof}
	Observe that $\G$ has maximal degree $\Delta$ almost surely because it is covered by $T$.
	By Lemma \ref{lem:3},
	\begin{align} \label{eqn:bound}
	& \frac{\E{\log |W|_{nk}(\G,o)} - \E{\log|N|_{nk}(\G,o)}}{nk} \;\; \geq \\ \nonumber
	& \qquad \qquad \Delta^{-2d} \, \E{\frac{1}{|N|_{nk}(\G,o)}\, \sum_{P \in N_{nk}(\G,o)} \frac{- \log ||M_k||(o,P_k)}{2k}}.
	\end{align}

	The expectation on the right hand side of \eqref{eqn:bound} is an average over $(\G,o,P^n)$,
	where $P^n$ is a uniformly random purely backtracking walk in $\G$ starting at $o$ and having length $nk$.
	If $k$ is even then $\pr{P^{n}_k = o} \geq \Delta^{-k}$.
	This is because a purely backtracking walk from $o$ of length $nk$ will be at $o$ at step $k$ if it is
	a purely backtracking walk from $o$ of length $k$ followed by one of length $nk - k$. Consequently,
	$$\pr{P^n_k = o} \geq \E{\frac{|N|_k(\G,o) \cdot |N|_{nk-k}(\G,o)}{|N|_{nk}(\G,o)}} \geq \Delta^{-k},$$
	where the last inequality is due to $|N|_k(\G,o) \geq 1$ and, also, by part 2 of Lemma \ref{lem:3}, due to
	$|N|_{nk-k}(\G,o) \geq \Delta^{-k} |N|_{nk}(\G,o)$. Since $- \log ||M_k||(u,v)$ is non-negative,
	\eqref{eqn:bound} implies that for every even $k$,
	$$ \frac{\E{\log |W|_{nk}(\G,o)} - \E{\log|N|_{nk}(\G,o)}}{nk} \geq \frac{\E{- \log ||M_{k}||(o,o)}}{2k \, \Delta^{2d+k}} \,.$$

	We may take a large $n$ limit supremum of the left hand side of the above for every even value of $k$.
	In the limit as $n \to \infty$, the left hand side is at most $\log \rho(\G) - \log \rho(T)$.
	This is because $\E{\log |W|} \leq \log \E{|W|}$ by concavity of $\log$ and,
	as argued in Lemma \ref{lem:Ramanujan}, $\E{\log |N|_{nk}(\G,o)}$ is the average over a
	finitely supported probability measure (on at most $m$ points) and each term in this average
	converges to $\log \rho(T)$ after division by $nk$ and letting $n$ tend to infinity.
	The inequality from the lemma now follows due to $k$ being an arbitrary even integer.
\end{proof}

\subsection{Completion of the proof}

Let $(H,v)$ be a rooted and connected graph. Given integers $k$ and $\ell$,
let us say $H$ contains a bouquet if it has two disjoint $\ell$-cycles $C_1$ and $C_2$ such that
if the distance from $v$ to $C_j$ is $r_j$, then $k \geq \ell + \max\{r_1,r_2\}$.
The situation is pictured below.
\begin{figure}[htpb]
\begin{center}
	\includegraphics[scale=0.5]{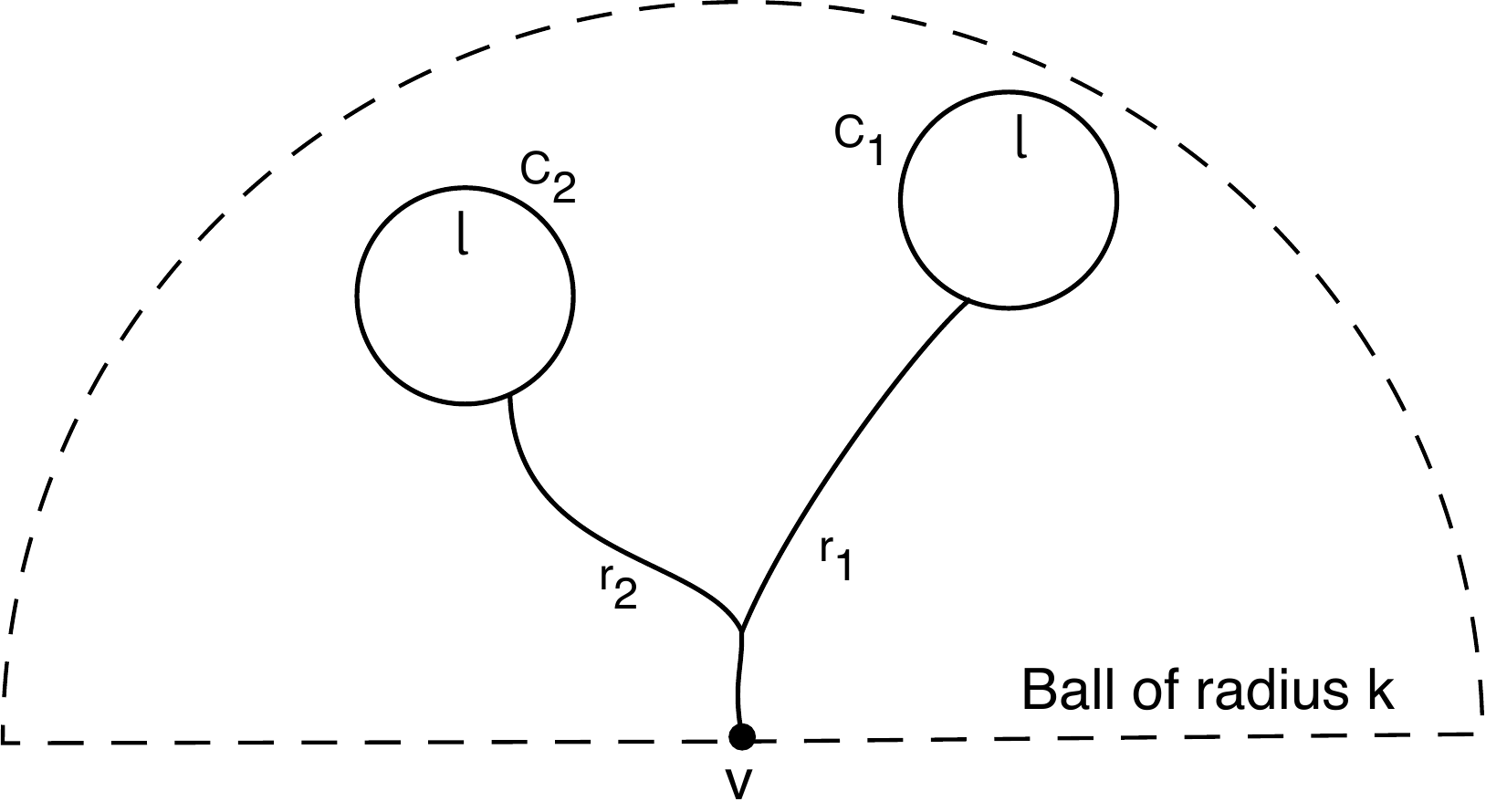}
	\caption{A bouquet around the root.}
	\label{fig:2}
\end{center}
\end{figure}

Suppose $(H,v)$ contains a bouquet for the parameter values $k$ and $\ell$.
They provide two closed walks in $W_{2k}(v,v)$, say $P_1$ and $P_2$, in the following way.
The walk $P_j$ is obtained by walking from $v$ to the closest vertex on $C_j$, traversing the cycle,
then walking back to $v$ along the reverse of the initial segment and appending some purely backtracking
walk at the end to ensure $2k$ steps in total.

Recall the walks in $W_{2k}(v,v)$ map to a set $\overline{W_{2k}(v,v)} \subset \pi (H,v)$ by homotopy equivalence.
The walks $P_1$ and $P_2$ then correspond to two mutually free elements in $\pi(H,v)$.
Let $\Gamma$ be the subgroup of $\pi(H,v)$ generated by $\overline{W_{2k}(v,v)}$.
It is a finitely generated free group of rank at least 2.
Recall that the uniform measure on $W_{2k}(v,v)$ pushes forward to a measure on $\overline{W_{2k}(v,v)}$,
which induces a symmetric random walk on $\Gamma$ whose Markov operator is denoted $M$.
The step distribution of the walk assigns positive probability to every element of $\overline{W_{2k}(v,v)}$.
So by Kesten's Theorem, specifically \cite[Corollary 3]{Ks}, the operator norm
\begin{equation} \label{eqn:Ks} || M || < 1. \end{equation}

As a result of \eqref{eqn:Ks}, the following lemma implies
\[\sup_{k \geq 1} \; \frac{\E{- \log ||M_{2k}||(o,o)}}{4k \, \Delta^{2d+2k}} > 0,\]
from which the proof of Theorem \ref{thm:1} follows by Lemma~\ref{lem:4}.

\begin{lemma}\label{lem:8}
	Let $(\G,o)$ be an infinite sofic graph such that for some $\ell$,
	\[\pr{o \;\text{lies in an}\; \ell-\text{cycle of} \; \G} > 0.\]
	Then there is a deterministic integer $k$ such that, with positive probability, $(\G,o)$ contains a bouquet with respect to
	the parameters $k$ and $\ell$.
\end{lemma}
\begin{proof}
	Let $N_R(v)$ be the number of $\ell$-cycles in a graph $H$ such that at least one of its vertices is within
	distance $R$ of vertex $v$. We will show below that for $(\G,o)$,
	$$\E{N_R(o)} \geq (R/ \ell)\, \pr{o \;\text{lies in an}\; \ell-\text{cycle of}\; \G}.$$
	Assuming this, we may choose an $R$ in terms of $\ell$ such that $\E{N_R(o)} \geq \ell \Delta^{\ell} + 2$.
	In this case, with positive probability, there are at least $\ell \Delta^{\ell} + 2$ different $\ell$-cycles
	in $\G$ within distance $R$ of the root $o$. Whenever this happens there must be two disjoint $\ell$-cycles
	within distance $R$ of the root, as we explain below. We may take $k$ to be $R + \ell$.

	The reason there are two disjoint $\ell$-cycles is that if $H$ has maximal degree
	$\Delta$, and $v$ is any vertex, then there can be at most $\Delta^{\ell}$ different $\ell$-cycles that
	pass through $v$. This means any specific $\ell$-cycle can meet at most $\ell \Delta^{\ell}$ other $\ell$-cycles.
	So when there are $\ell \Delta^{\ell} + 2$ different $\ell$-cycles, some two among them are disjoint.

	In order to get the lower estimate on $\E{N_R(o)}$ consider the function
	$$ F(H,u,v) = \mathbf{1}\left \{ \mathrm{dist}(u,v) \leq R\; \text{and}\; u\; \text{lies in an}\; \ell-\text{cycle of}\; H \right\}.$$
	Then,
	$$ \sum_{u \in H} F(H,u,v) =  \# \, \{\text{vertices in}\; \ell-\text{cycles of}\; H
	\; \text{within distance}\; R\; \text{of}\; v\} \leq \ell N_R(v),$$
	and
	$$ \sum_{v \in H} F(H,u,v) = |B_R(H,u)| \, \mathbf{1} \{ u \; \text{lies in an}\; \ell-\text{cycle of}\; H\}.$$
	Since $(\G,o)$ is infinite almost surely, $|B_R(\G,o)| \geq R$. The mass transport principle then provides
	the lower bound on $\E{N_R(o)}$ as displayed above.
\end{proof}

\section{A spectral gap theorem for finite graphs} \label{sec:finite}

In this section we prove Theorem \ref{thm:2}. Let $G$ be a finite and connected
graph with universal cover $T$ and cover map $\pi$. Since $\lm_1(G)$ is also the largest
eigenvalue of $G$ in absolute value, we denote it $\rho(G)$ henceforth.

For a graph $H$ and $x \in \ell^2(H)$, let
\begin{equation} \label{eqn:f}
f_H(x) = 2 \sum_{\{u,v\}\in H} x_ux_v\, ,
\end{equation}
where the summation is over the edges of $H$ counted with multiplicity as there may
be multi-edges and loops (recall a loop contributes degree 2 to its vertex). Thus,
\begin{equation*}
\rho(T) = \sup_{\substack{x \in \ell^2(T) \\ ||x|| = 1}} |f_T(x)|
\quad
\text{and}
\quad
\rho(G) = \sup_{\substack{x \in \ell^2(G) \\ ||x|| = 1}} |f_G(x)|.
\end{equation*}

Theorem \ref{thm:2} follows from the Propositions \ref{prop:unicyclic} and \ref{prop:multicyclic} given below.

\subsection{Spectral radius of an unicyclic graph} \label{sec:finite-A}

\begin{prop} \label{prop:unicyclic}
	Let $G$ be a finite and connected graph with at most one cycle. Then $\rho(G) = \rho(T)$.
\end{prop}

\begin{proof}
	There is nothing to prove if $G$ is a tree, so assume that $G$ has exactly one cycle
	(possibly a loop, or a 2-cycle made by a pair of multi-edges).
	We give an explicit description of $T$ in terms of $G$.

	Let the unique cycle in $G$ consist of vertices $v_1,\dots,v_n$, in that order.
	Let $H$ be the graph obtained by deleting edge $(v_n,v_1)$ from $G$.
	We construct countably infinite copies $\dots, H_{-1}, H_0, H_1, \dots$ of $H$, indexed by $\Z$.
	For each $k\in \Z$, we draw an edge between $v_n$ in $H_k$ and $v_1$ in $H_{k+1}$.
	The resulting graph is $T$.

	Let $y \in \ell^2(G)$ be the maximal eigenvector of $G$, normalized to $||y||=1$ and with positive entries.
	Thus, $f_G(y)=\rho(G)$. We will construct an $x\in \ell^2(T)$, with $||x|| = 1$, such that $f_T(x)$
	approximates $\rho(G)$ arbitrarily closely.

	Fix an arbitrary $N\in \N$.
	For $v'\in H_1,\dots,H_N$, set $x_{v'} = \frac{1}{\sqrt{N}} y_{\pi(v')}$.
	For all other $v'\in T$, set $x_{v'} = 0$.

	It is evident that $||x|| = ||y|| = 1$.  Moreover,
	\[f_T(x) = 2\sum_{\substack{(u',v')\in T \\ u',v'\in H_1\cup \dots \cup H_N} } \frac{1}{N} y_{\pi(u')}y_{\pi(v')}.\]
	For each edge $(u,v)\in G$ the term $\frac{1}{N} y_{u}y_{v}$ appears $N$ times in the above sum,
	except for $\frac{1}{N} y_{v_1}y_{v_n}$, which appears $N-1$ times. Therefore,
	\[f_T(x) = 2\sum_{(u,v)\in G} y_uy_v - \frac{2}{N} y_{v_1}y_{v_n} = \rho(G) - \frac{2}{N}y_{v_1}y_{v_n}.\]
	As $N$ was arbitrary, the error term $\frac{2}{N}y_{v_1}y_{v_n}$ can be made arbitrarily small.
\end{proof}

\subsection{Spectral gap for a multi-cyclic graph} \label{sec:finite-B}

\begin{prop} \label{prop:multicyclic}
	Let $G$ be a finite and connected graph with at least two cycles.
	Then $\rho(T) < \rho(G)$.
\end{prop}

For the remainder of this section we assume $G$ is a finite and connected graph with
at least two cycles (which may intersect, may be loops, or cycles made by multi-edges).

The \textit{2-core} of $G$ is defined by the following procedure.
If $G$ has at least one leaf, pick an arbitrary leaf and delete it.
This operation may produce more leaves.
Repeat the leaf removal operation until there are no leaves.
The resulting subgraph of $G$ is its 2-core.

The 2-core of a graph is non-empty if and only if it contains a cycle. Moreover,
all cycles are preserved in its 2-core. Consequently, since $G$ has two distinct
cycles, so does its 2-core.

Let $G_{\text{int}}$ denote the 2-core of $G$.
Let $V^G_{\mathrm{int}}$ denote the vertices of $G_{\text{int}}$.
Let $E^G_{\mathrm{int}}$ be the edges of $G_{\text{int}}$ \emph{directed}
both ways, so that every edge $\{u,v\} \in G_{\text{int}}$ becomes two
directed edges $(u,v)$ and  $(v,u)$ in $E^G_{\mathrm{int}}$.

Denote by $V^G_{\mathrm{ext}}$ the vertices of $G\setminus G_{\text{int}}$.
Let $E^G_{\mathrm{ext}}$ be the edges of $G\setminus G_{\text{int}}$ such that
they are \emph{directed away} from the 2-core. This is possible because for
every edge $\{u,v\}$ in $G\setminus G_{\text{int}}$, there is a unique shortest
path from $G_{\text{int}}$ that terminates at $\{u,v\}$. The orientation of
$\{u,v\}$ is then in the direction this path enters the edge.
The figure below gives an illustration of these definitions.

\begin{figure}[hbtp]
\centering
\includegraphics{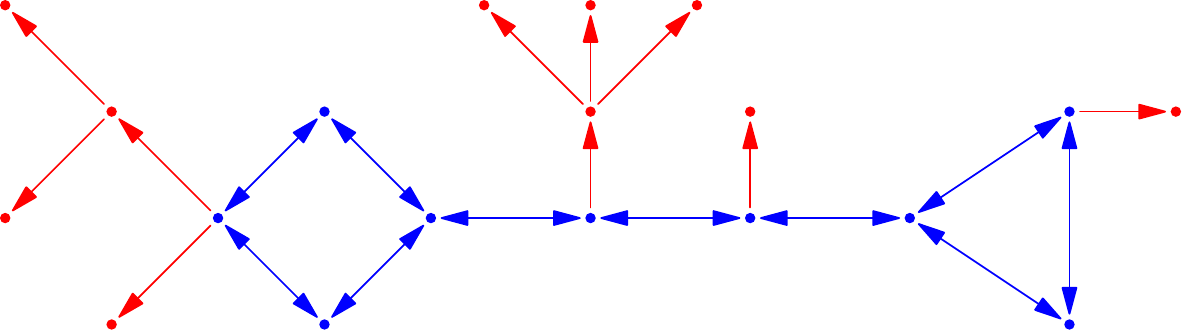}
\caption{ An example illustrating the definitions of
$V^G_{\mathrm{int}}$, $V^G_{\mathrm{ext}}$, $E^G_{\mathrm{int}}$, and $E^G_{\mathrm{ext}}$.
$V^G_{\mathrm{int}}$ and $E^G_{\mathrm{int}}$ are coloured blue.
$V^G_{\mathrm{ext}}$ and $E^G_{\mathrm{ext}}$ are coloured red.}
\end{figure}

\begin{lemma}\label{lem:constant-factors-delta}
There exists a function  $\Gamma: E^G_{\mathrm{int}} \to [1, 2)$
such that for each directed edge $(u,v) \in E^G_{\mathrm{int}}$,
\[\sum_{\substack{w:~(v,w)\in E^G_{\mathrm{int}} \\ w\neq u}} \Gamma(v,w) > \Gamma(u,v).\]
\end{lemma}
Note $\Gamma$ is not symmetric, i.e. $\Gamma(u,v)$ need not equal $\Gamma(v,u)$.

\begin{proof}
Every vertex of $G_{\text{int}}$ has degree at least $2$ within this subgraph.
The following property of $G_{\text{int}}$ is crucial: since $G_{\text{int}}$ has at least two cycles and is connected,
every cycle of $G_{\text{int}}$ contains a vertex of degree more than 2.

For each directed edge $(u,v) \in E^G_{\mathrm{int}}$ with $\deg u > 2$, set $\Gamma(u,v) = 1$.
The remaining values of $\Gamma(u,v)$ will correspond to directed edges $(u,v)$ with $\deg u = 2$.
We assign these values by the following iterative procedure.

Fix an $\epsilon>0$ to be determined. If $\deg u = 2$, $u$ is adjacent to $v_1$ and $v_2$,
and $\Gamma(v_1,u)$ has been assigned, assign $\Gamma(u,v_2) = \Gamma(v_1,u) + \epsilon$.
Due to the aforementioned crucial property, this procedure assigns a value of the form $1+m\epsilon$ to every
$\Gamma(u,v)$ with $(u,v) \in E^G_{\mathrm{int}}$. Finally, since $G$ is finite we may choose $\epsilon$
small enough such that $\Gamma$ is strictly less than 2 everywhere.

Now if $(u,v)\in E^G_{\mathrm{int}}$ and $\deg v > 2$,
$$\sum_{\substack{w:~(v,w)\in E^G_{\mathrm{int}} \\ w\neq u}} \Gamma(v,w) \ge 2 > \Gamma(u,v).$$
If $(u,v)\in E^G_{\mathrm{int}}$ and $\deg v = 2$,
$$\sum_{\substack{w:~(v,w)\in E^G_{\mathrm{int}} \\ w\neq u}} \Gamma(v,w) = \Gamma(u,v) + \epsilon > \Gamma(u,v).$$
\end{proof}

\begin{lemma}\label{lem:constant-factors-gamma}
There exists a function $\Delta: E^G_{\mathrm{ext}} \to (0,1]$
such that for each directed edge $(u,v) \in E^G_{\mathrm{ext}}$,
\[\sum_{\substack{w:~(v,w)\in E^G_{\mathrm{ext}}}} \Delta(v,w) < \Delta(u,v).\]
\end{lemma}

\begin{proof}
The edges in $E^G_{\mathrm{ext}}$ form trees, rooted at vertices in $V^G_{\mathrm{int}}$ and directed toward the leaves.

For each edge $(u,v)\in E^G_{\mathrm{ext}}$, where $u\in V^G_{\mathrm{int}}$, set $\Delta(u,v) = 1$.
Assign the remaining variables by recursing down the trees in the following way.
If $\Delta(u,v)$ has been assigned and $v$ has $d$ out-edges $(v,w)\in E^G_{\mathrm{ext}}$,
set $\Delta(v,w) = \frac{1}{d+1}\Delta(u,v)$ for each out-edge $(v,w)$.
Then,
$$\sum_{\substack{w:~(v,w)\in E^G_{\mathrm{ext}}}} \Delta(v,w) = \frac{d}{d+1} \Delta(u,v) < \Delta(u,v),$$
so the desired inequality holds.
\end{proof}

\paragraph{\textbf{Proof of Proposition \ref{prop:multicyclic}}}
Let $y$ be the eigenvector of the maximal eigenvalue of $G$ chosen
such that all its entries are positive and $|| y || = 1$.
Note this identity for every vertex $u\in G$:
\begin{equation}\label{eqn:eigenvector}
\sum_{\substack{v:~\{u,v\}\in G}} \frac{y_{v}}{y_{u}} = \rho(G).
\end{equation}

Root $T$ at any vertex $r$ such that $\pi(r)\in V^G_{\mathrm{int}}$.
For the rest of this proof, when we refer to an edge $(u,v)\in T$ the first vertex $u$ is the parent,
that is, closer to the root than $v$.

Let $V^T_{\mathrm{int}}$ be the vertices in $T$ with infinitely many descendants,
and $V^T_{\mathrm{ext}}$ be the vertices in $T$ with finitely many descendants.
Let $E^T_{\mathrm{int}}$ denote the edges $(u,v)\in T$ with $v\in V^T_{\mathrm{int}}$,
and $E^T_{\mathrm{ext}}$ the edges $(u,v)\in T$ with $v\in V^T_{\mathrm{ext}}$.

Observe that $u\in V^T_{\mathrm{int}}$ (resp. $V^T_{\mathrm{ext}}$)
if and only if $\pi(u)\in V^G_{\mathrm{int}}$ (resp. $V^G_{\mathrm{ext}}$).
Similarly, $(u,v)\in E^T_{\mathrm{int}}$ if and only if $(\pi(u),\pi(v)) \in E^G_{\mathrm{int}}$,
and $(u,v)\in E^T_{\mathrm{ext}}$ if and only if $(\pi(u),\pi(v)) \in E^G_{\mathrm{ext}}$.
In the latter case it is crucial that $u$ is the parent of $v$;
this requires $v$ to be farther than $u$ from $V^T_{\mathrm{int}}$,
so $\pi(v)$ is farther than $\pi(u)$ from $V^G_{\mathrm{int}}$.
Thus $(\pi(u),\pi(v))$ has the necessary orientation of an edge in $E^G_{\mathrm{ext}}$.

Consider the functions $\Gamma$ and $\Delta$ from
Lemmas~\ref{lem:constant-factors-delta} and \ref{lem:constant-factors-gamma}.
Let $\gamma, \delta >0$ be (small) constants to be determined later.
Throughout the following argument we will use that
$$2|ab| \leq \eta a^2 + \eta^{-1} b^2 \quad \text{for}\;\; \eta > 0.$$

For each edge $(u,v)\in E^T_{\mathrm{int}}$, we have
\begin{equation}\label{eqn:estimate-int}
2|x_ux_v|
\le
\frac{y_{\pi(v)}}{y_{\pi(u)}}\left(1 + \frac{\Gamma(\pi(u),\pi(v))\gamma}{y_{\pi(u)}y_{\pi(v)}}\right)^{-1}x_u^2
+
\frac{y_{\pi(u)}}{y_{\pi(v)}}\left(1 + \frac{\Gamma(\pi(u),\pi(v))\gamma}{y_{\pi(u)}y_{\pi(v)}}\right)x_v^2.
\end{equation}
Analogously, for each edge $(u,v)\in E^T_{\mathrm{ext}}$,
\begin{equation}\label{eqn:estimate-ext}
2|x_ux_v|
\le
\frac{y_{\pi(v)}}{y_{\pi(u)}}\left(1 + \frac{\Delta(\pi(u),\pi(v))\delta}{y_{\pi(u)}y_{\pi(v)}}\right)x_u^2
+
\frac{y_{\pi(u)}}{y_{\pi(v)}}\left(1 + \frac{\Delta(\pi(u),\pi(v)) \delta}{y_{\pi(u)}y_{\pi(v)}}\right)^{-1}x_v^2.
\end{equation}
The quantity $\Delta(\pi(u),\pi(v))$ is defined because $(\pi(u),\pi(v))$ has the correct orientation of an edge in $E^G_{\mathrm{ext}}$, as noted above.

Recall $f_T$ from \eqref{eqn:f}. The estimates (\ref{eqn:estimate-int}) and (\ref{eqn:estimate-ext}) imply that
\begin{equation}\label{eqn:main-bound}
|f_T(x)| \leq 2 \sum_{\{u,v\} \in T} |x_ux_v| \le \sum_{u\in T}g(u) x_u^2,
\end{equation}
where $g(u)$ is as follows.
Let $\pa(u)$ denote the parent of vertex $u \in T$ and $\ch(u)$ denote the set of all children of $u$.
If $u\in V^T_{\mathrm{int}}$ then
\begin{align*}\label{eqn:def-g-int}
\begin{split}
g(u)
&= \frac{y_{\pi(\pa(u))}}{y_{\pi(u)}} \left(1 + \frac{\Gamma(\pi(\pa(u)),\pi(u)) \gamma}{y_{\pi(\pa(u))}y_{\pi(u)}}\right)
+ \sum_{c\in \ch(u) \cap V^T_{\mathrm{int}}} \frac{y_{\pi(c)}}{y_{\pi(u)}}\left(1 + \frac{\Gamma(\pi(u),\pi(c))\gamma}{y_{\pi(u)}y_{\pi(c)}}\right)^{-1}
\\
&+ \sum_{d\in \ch(u) \cap V^T_{\mathrm{ext}}} \frac{y_{\pi(d)}}{y_{\pi(u)}}\left(1 + \frac{\Delta(\pi(u),\pi(d)) \delta}{y_{\pi(u)}y_{\pi(d)}}\right).
\end{split}
\end{align*}
If $u\in V^T_{\mathrm{ext}}$ then
\begin{equation*}\label{eqn:def-g-ext}
g(u)
= \frac{y_{\pi(\pa(u))}}{y_{\pi(u)}} \left(1 + \frac{\Delta (\pi(\pa(u)),\pi(u)) \delta}{y_{\pi(\pa(u))}y_{\pi(u)}}\right)^{-1}
+ \sum_{d\in \ch(u)} \frac{y_{\pi(d)}}{y_{\pi(u)}}\left(1 + \frac{\Delta (\pi(u),\pi(d)) \delta}{y_{\pi(u)}y_{\pi(d)}}\right).
\end{equation*}

Due to \eqref{eqn:main-bound}, the proposition will be proved by
showing that $g(u)$ is uniformly bounded away from $\rho(G)$ over all vertices $u$.
We separately consider the two cases $u\in V^T_{\mathrm{int}}$ and $u\in V^T_{\mathrm{ext}}$.

Suppose $u\in V^T_{\mathrm{int}}$.
Then for all sufficiently small $\gamma>0$ we have the bound
\begin{equation}\label{eqn:inv-estimate}
\sum_{c\in \ch(u) \cap V^T_{\mathrm{int}}}
\frac{y_{\pi(c)}}{y_{\pi(u)}}
\left(1 + \frac{\Gamma(\pi(u),\pi(c))\gamma}{y_{\pi(u)}y_{\pi(c)}}\right)^{-1}
\le
\sum_{c\in \ch(u) \cap V^T_{\mathrm{int}}}
\frac{y_{\pi(c)}}{y_{\pi(u)}}
\left(1 - \frac{\Gamma(\pi(u),\pi(c))\gamma}{y_{\pi(u)}y_{\pi(c)}}\right)
+ C_u\gamma^2,
\end{equation}
for some constant $C_u \geq 0$ depending on $u$.

The terms in (\ref{eqn:inv-estimate}) depend only on the vertices $\pi(u)$ and $\pi(c)$ for $c \in \ch(u)$.
These are vertices of $G$ and, since $G$ is finite, there are only finitely many
distinct values of $C_u$. Let $C$ be the maximum of the $C_u$s. In the inequality \eqref{eqn:inv-estimate}
we may replace every $C_u$ by $C$, as we do henceforth.

Inequality \eqref{eqn:inv-estimate} implies the following
bound for every $u\in V^T_{\mathrm{int}}$ and all sufficiently small $\gamma > 0$.
\begin{align}\label{eqn:estimate-g-int}
\begin{split}
g(u)
&\le
\rho(G)
+ \frac{\gamma}{y_{\pi(u)}^2}
\left[\Gamma(\pi(\pa(u)),\pi(u)) - \sum_{c\in \ch(u) \cap V^T_{\mathrm{int}}} \Gamma(\pi(u),\pi(c))\right] \\
&+ \frac{\delta}{y_{\pi(u)}^2}
\sum_{d\in \ch(u) \cap V^T_{\mathrm{ext}}} \Delta(\pi(u),\pi(d))
+ C\gamma^2.
\end{split}
\end{align}
This is obtained by substituting \eqref{eqn:inv-estimate} into the definition of $g(u)$,
then multiplying out the terms and simplifying the sums by using the eigenvector
equation \eqref{eqn:eigenvector}.

By Lemma \ref{lem:constant-factors-delta},
\begin{equation*}
\Gamma(\pi(\pa(u)),\pi(u)) - \sum_{c\in \ch(u) \cap V^T_{\mathrm{int}}} \Gamma(\pi(u),\pi(c)) < 0
\end{equation*}
for every $u\in V^T_{\mathrm{int}}$. Moreover, as $u$ ranges over $ V^T_{\mathrm{int}}$ the quantities
\begin{equation*}
u \mapsto \; \frac{1}{y_{\pi(u)}^2}
\left[\Gamma(\pi(\pa(u)),\pi(u)) - \sum_{c\in \ch(u) \cap V^T_{\mathrm{int}}} \Gamma(\pi(u),\pi(c))\right]
\end{equation*}
are determined by the graph $G$. So they attain finitely many values and have
a maximum value $C_{\text{int}} < 0$. Analogously, the quantities
\begin{equation}
u \mapsto \; \frac{1}{y_{\pi(u)}^2}
\sum_{d\in \ch(u) \cap V^T_{\mathrm{ext}}} \Delta(\pi(u),\pi(d))
\end{equation}
have a maximum value $D_{\text{int}} \geq 0$ as $u$ ranges over $V^T_{\mathrm{int}}$.
So we infer that for every $u\in V^T_{\mathrm{int}}$ and all sufficiently small $\gamma>0$,
\begin{equation}\label{eqn:estimate-g-int-final}
g(u) \le \rho(G) + C_{\text{int}} \gamma + C\gamma^2 + D_{\text{int}} \delta.
\end{equation}

Now suppose that $u\in V^T_{\mathrm{ext}}$.
By an analogous argument as above, there exists a constant $D \geq 0$ independently of $u$
such that for all sufficiently small $\delta>0$,
\begin{equation*}\label{eqn:estimate-g-ext}
g(u) \le \rho(G) +
\frac{ \delta}{y_{\pi(u)}^2} \left[ - \Delta(\pi(\pa(u)),\pi(u)) + \sum_{d\in \ch(u)} \Delta(\pi(u),\pi(d)) \right] + D\delta^2.
\end{equation*}
By Lemma \ref{lem:constant-factors-gamma},
\begin{equation*}
- \Delta(\pi(\pa(u)),\pi(u)) + \sum_{d\in \ch(u)} \Delta(\pi(u),\pi(d)) < 0
\end{equation*}
for all $u\in V^T_{\mathrm{ext}}$.
Therefore, as before, there is a $D_{\text{ext}} < 0$ such that for every $u\in V^T_{\mathrm{ext}}$
and all sufficiently small $\delta>0$,
\begin{equation}\label{eqn:estimate-g-ext-final}
g(u) \le \rho(G) + D_{\text{ext}} \delta + D\delta^2.
\end{equation}

Finally, we select $\gamma>0$ small enough that (\ref{eqn:estimate-g-int-final}) holds and
$C_{\text{int}}\gamma + C\gamma^2<0$. This is possible because $C_{\text{int}}<0$.
Then we select $\delta>0$ small enough such that (\ref{eqn:estimate-g-ext-final}) holds while
both $C_{\text{int}} \gamma + C\gamma^2 + D_{\text{int}} \delta < 0$ and
$D_{\text{ext}} \delta + D\delta^2 < 0$. This is possible due to the choice of $\gamma$
and because $D_{\text{ext}}<0$.

In light of (\ref{eqn:estimate-g-int-final}) and (\ref{eqn:estimate-g-ext-final}),
our choice of $\gamma$ and $\delta$ above imply that there is an $\epsilon > 0$
such that for every vertex $u \in T$, $g(u) \le \rho(G)-\epsilon$. This implies $\rho(T) < \rho(G)$.
\qed

\section{Future directions} \label{sec:conc}
It would be interesting to find an effective version of Theorem \ref{thm:0} in the following sense.
Let $G_1, G_2, G_3, \ldots$ be finite, connected graphs with $|G_n| \to \infty$.
Suppose they have a common universal cover $T$, and are Ramanujan in that all but their largest
eigenvalue are at most $\rho(T)$ in absolute value. What is the ``essential girth" of $G_n$ in terms
of its size, which is to say, the asymptotic girth after possibly having removed at order of $o(|G_n|)$ edges?
For $d$-regular Ramanujan graphs it is known, see \cite{AGV}, that the essential girth is at least
of order $\log \log |G|$. All known constructions provide graphs with girth of order $\log |G|$.
It seems that a lower bound of order $\log |G|$ for the girth is unknown even when $G$ is a
Cayley graph.

It would also be interesting to find an effective form of Theorem \ref{thm:2} in terms of
the size and the maximal degree of $G$. The theorem is in some ways an analogue of
Theorem \ref{thm:1} for finite graphs, although, its word-for-word reformulation is false
for infinite graphs. In this regard it would be interesting to prove a spectral gap
between $\rho(G)$ and $\rho(T)$ under natural hypotheses on an infinite graph $G$.
For instance, to prove an effective spectral gap when there is an $R$ such that the
$R$-neighbourhood of every vertex in $G$ contains a cycle.

\section*{Acknowledgements}

The authors gratefully acknowledge the MIT Undergraduate Research Opportunities
Program in which some of this work was completed.  The first author thanks
Ryan Alweiss for many helpful conversations over the course of this work.
Thanks also to an anonymous referee for suggestions leading to various simplifications.

\end{document}